\documentclass[12pt]{article}
\usepackage[english]{babel}
\usepackage{amsmath, amsthm, amsfonts, amssymb, enumerate,fancyhdr}

 \newtheorem{thm}{Theorem}[section]
 \newtheorem{cor}[thm]{Corollary}
 \newtheorem{lem}[thm]{Lemma}
 
\title{Essential norm estimates for weighted composition operator on the logarithmic Bloch space}

\author{Mar{\'i}a T. MALAV\'E-RAM{\'I}REZ$^1$ and
Julio C. RAMOS-FERN\'ANDEZ$^2$\\
\small Departamento de Matem\'aticas\\
\small Universidad de Oriente\\
\small 6101 Cuman\'a, Estado Sucre\\
\small Rep\'ublica Bolivariana de Venezuela\\
\small $^1$mtmalave@udo.edu.ve and $^2$jcramos@udo.edu.ve}

\date{}

\begin{document}
\maketitle


\begin{abstract}
In this article, we estimate the essential norm of weighted composition operator  $W_{u, \varphi}$, acting on the logarithmic Bloch space  $\mathcal{B}^{v_{\log}}$, in terms of the  $n$-power of the analytic function  $\varphi$ and the norm of the  $n$-power of the identity function. Also, we estimate the essential norm of the weighted composition operator from $\mathcal{B}^{v_{\log}}$ into the growth space  $H_{v_{\log}}^{\infty}$.  As a consequence of our result, we estimate the essential norm of the composition operator $C_\varphi$ acting on the Logarithmic-Zygmund space.

\vspace{0.2cm}

\end{abstract}



\section{Introduction}
Let $\mathbb{D}$ be the unit disk of the complex plane  $\mathbb{C}$ and let  $H(\mathbb{D})$ be the space of all holomorphic functions on  $\mathbb{D}$ endowed with the topology of the uniform convergence on compact subsets of  $\Bbb D$. For fixed holomorphic functions $u:\Bbb D\to \Bbb C$ and  $\varphi:\Bbb D\to\Bbb D$, we can define the linear operator $W_{u,\varphi}:H(\Bbb D)\to H(\Bbb D)$ by
$$
 W_{u,\varphi}(f):= u\cdot( f\circ \varphi).
$$
Which is known as the {\it weighted composition} with symbols $u$ and $\varphi$. Clearly, if $u\equiv 1$ we have $W_{1,\varphi}(f)=  f\circ \varphi =C_\varphi(f)$, the composition operator $C_\varphi$, and if $\varphi(z)=id(z)=z$ for all $z\in\Bbb D$, we have obtain $W_{u,id}(f)= u\cdot f=M_u(f)$,  the multiplication operator $M_u$. Furthermore, we can see that
 $W_{u,\varphi}$ is 1-1 on $H(\Bbb D)$ unless that $u\equiv 0$ or $\varphi$ is a constant function. However, if we wish to study properties like as continuity, compactness, essential norm, etc. of this operator, we need restrict the domain and target space $H(\Bbb D)$ to a normed and complete subspace of $H(\Bbb D)$. In this article we consider the restriction to the growth space of all analytic functions $f$ on $\Bbb D$ such that
\begin{equation}\label{def-norma-Hv}
\|f \|_{H_v^\infty}= \sup_{z \in \mathbb{D}}v(z)|f(z)| <\infty,
\end{equation}
where $v:\Bbb D\to \Bbb R^+$ is a {\it weight function}, that is, a bounded, continuous and positive function defined on $\Bbb D$. Also, we consider the Bloch-type space $\mathcal{B}^v$ of all analytic function $f$ on $\Bbb D$ such that $f'\in H_v^\infty$. It is known that $H_v^\infty$ is a Banach space with the norm defined in (\ref{def-norma-Hv}) and  $\mathcal{B}^v$  is a Banach space with the norm
$$
\|f\|_{\mathcal{B}^v} =\left|f(0)\right| + \|f '\|_{H_v^\infty} = \left|f(0)\right| +\|f \|_{\widetilde{\mathcal{B}}^{v}},
$$
where,
$$
\|f \|_{\widetilde{\mathcal{B}}^{v}}= \sup_{z \in \mathbb{D}}v(z)|f'(z)|.
$$
The properties of $W_{u,\varphi}$ acting between growth-type spaces were studied by Hyv\"arinen et al. \cite{HKLRS12} and by Malav\'e-Ram{\'i}rez and Ramos-Fern\'andez \cite{malave-ramos}, for very general weights $v$; however, properties of $W_{u,\varphi}$ acting on Bloch-type spaces are  still in develops. About this last, we can mention the works of Hyv\"{a}rinen and Lindstr\"{o}m in \cite{hyvarinen-lindstrom}.  Also,  there are  no much works about the properties of $W_{u,\varphi}$ between $H_v^\infty$ and $\mathcal{B}^v$, we can mention the work of Stevi\'c in \cite{stevic}.

\vspace{0.2cm}
In this note, we estimate the essential norm of $W_{u,\varphi}$ acting on the logarithmic Bloch space $\mathcal{B}^{v_{\log}}$ (also known as the weighted Bloch space), where the weight consider here is defined by
$$
  v_{\log}(z)=(1-|z|)\log \left(\frac{2}{1-|z|} \right)
$$
with $z\in\Bbb D$ which clearly is radial and typical ($\lim_{|z|\to 1^-}v_{\log}(z)=0$). This  space appears in the literature when we study properties of certain operators acting on certain spaces of analytic functions on the unit disk $\Bbb D$ of the complex plane $\Bbb C$. For instance, in 1991,  Brown and Shields \cite{BS91} showed that an analytic function $u$ is a multiplier on the Bloch space $\mathcal{B}$ if and only if $u\in \mathcal{B}^{v_{\log}}$.
Also,  in 1992,    K. Attele \cite{At92} showed  that the Hankel operator induced by a function $f\in H(\Bbb D)$ in the Bergman space $L_a^1$ (for the definition of Bergman space and the Hankel operator see  \cite{Zh90}) is bounded if and only if $\left\|f\right\|_{\mathcal{B}^{v_{\log}}}<\infty$. The study of composition operators acting on the weighted Bloch space began with the work of Yoneda \cite{yoneda}, where he characterized the continuity
and compactness of composition operators acting on the weighted Bloch space $\mathcal{B}^{v_{\log}}$. These last results  were extended by  Galanoupulus  \cite{galanopoulus} and Ye \cite{ye} for weighted composition operators acting on $\mathcal{B}^{v_{\log}}$. More recently,
Malav\'e-Ram{\'i}rez and Ramos-Fern\'andez  \cite{malave-ramos}, following similar ideas used by  Hyv\"{a}rinen and  Lindstr\"{o}m in  \cite{hyvarinen-lindstrom}, characterized the continuity and compactness of  $W_{u, \varphi}$ acting on $\mathcal{B}^{v_{\log}}$ in terms of certain expression involving the  $n$-th power of $\varphi$ and
the log-Bloch norm of the $n$-th power of the identity function on $\Bbb D$. The results obtained by Malav\'e-Ram{\'i}rez and Ramos-Fern\'andez can be enunciated as follows:

\begin{thm}[\cite{malave-ramos}]\label{cor2}
Suppose that  $u:\Bbb D\to\Bbb C$ and  $\varphi:\Bbb D\to\Bbb D$ are holomorphic functions.
  \begin{enumerate}
   \item The operator  $W_{u, \varphi}$  is continuous on  $\mathcal{B}^{v_{\log}}$  if and only if
    $$
     \max\left\{\sup_{n\in\Bbb W}\frac{(n+1)\left\|J_u\left(\varphi^n\right)\right\|_{\mathcal{B}^{v_{\log}}}}
  {\left\|g_{n+1}\right\|_{\mathcal{B}^{w_{\log}}}},\sup_{n\in\Bbb W}\frac{\left\|I_u\left(\varphi^n\right)\right\|_{\mathcal{B}^{v_{\log}}}}
  {\left\|g_n\right\|_{\mathcal{B}^{v_{\log}}}}\right\}<\infty.
    $$
  \item The operator  $W_{u, \varphi}$  is compact on $\mathcal{B}^{v_{\log}}$ if and only if
  $$
  \max \left\{\lim_{n \to \infty}\dfrac{(n+1)\| J_{u}(\varphi^n)\|_{\mathcal{B}^{v_{\log}}}}{\|g_{n+1}\|_{\mathcal{B}^{w_{\log}}}}, \lim_{n \to \infty} \dfrac{\|I_{u}(\varphi^n)\|_{\mathcal{B}^{v_{\log}}}}{\|g_n\|_{\mathcal{B}^{v_{\log}}}} \right\}=0,
  $$
  \end{enumerate}
where $\Bbb W=\left\{0,1,2,\cdots\right\}$, $g_0\equiv 1$, $g_n(z)=z^n$ for $n\in\Bbb N=\left\{1,2,3,\cdots\right\}$ and $z\in\Bbb D$,
\begin{equation}\label{def-peso-w}
w_{\log}(z)=\left[\log \log \left(\dfrac{4}{1-|z|^2} \right)\right]^{-1},
\end{equation}
and the functionals  $I_u,J_u:H(\Bbb D)\to \Bbb C$ are defined by
 $$
  I_u\left(f(z)\right)= \int_0^z f'(s) u(s) ds,\hspace{0.2cm}\text{and}\hspace{0.2cm} J_u\left(f(z)\right)= \int_0^z f(s) u'(s) ds.
 $$
\end{thm}

 The main goal of the present article is to find an estimation of the essential norm of the operator  $W_{u, \varphi}:\mathcal{B}^{v_{\log}}\to \mathcal{B}^{v_{\log}}$  which implies the result, about compactness, mentioned in the item (2) of Theorem \ref{cor2} above. Allow us recall that the {\it essential norm} of a continuous  operator $T:X\to Y$, between Banach spaces $X$ and $Y$, denoted by $\| T\|_{e}^{X\to Y}$, is its distance to the class of the compact operators, that is, $\| T\|_{e}^{X\to Y}=\inf \{\|T-K \|^{X\to Y}~:~ K:X\to Y~ \textrm{is ~ compact} \}$, where $\|T\|^{X\to Y}$ denotes the norm of the operator $T:X\to Y$. Notice that $T:X\to Y$ is compact if and only if $\| T\|_{e}^{X\to Y}=0$.

\vspace{0.2cm}
 Recent results about essential norm estimates on log-Bloch spaces can be found in \cite{CCFR13} for the composition operator $C_{\varphi}: \mathcal{B}^{v_{3}}\to \mathcal{B}^{v}$, with $v_{3}(z)=(1-|z|)\log\left(\frac{3}{1-|z|} \right)$ and in \cite{ye2}, where Ye estimated the essential norm of the operator $DC_{\varphi}: \mathcal{B}^{v_{e}}\to H_{v}^{\infty}$ defined by $DC_{\varphi}(f):=W_{\varphi',\varphi}(f')$ with
$v_{e}(z)=(1-|z|)\log \left(\frac{2e}{1-|z|} \right)$.

\vspace{0.2cm}
 In this article we are going to show the following result:

\begin{thm}\label{teo1}
Suppose that $u:\Bbb D\to \Bbb C$ and  $\varphi:\Bbb D\to \Bbb D$ are holomorphic functions and that  $ W_{u,\varphi}:\mathcal{B}^{v_{\log}}\to \mathcal{B}^{v_{\log}}$ is continuos. Then
$$
\|W_{u, \varphi} \|_{e}^{\mathcal{B}^{v_{\log}}\to \mathcal{B}^{v_{\log}}}
\simeq  \max \left\{\limsup_{n \to \infty}\dfrac{(n+1)\| J_{u}(\varphi^n)\|_{\mathcal{B}^{v_{\log}}}}{\|g_{n+1}\|_{\mathcal{B}^{w_{\log}}}}, \limsup_{n \to \infty} \dfrac{\|I_{u}(\varphi^n)\|_{\mathcal{B}^{v_{\log}}}}{\|g_n\|_{\mathcal{B}^{v_{\log}}}} \right\}.
$$
\end{thm}

 Above, and in what follows, for two positive quantities $A$ and $B$, we write $A\simeq B$ and say that $A$ {\it is equivalent to} $B$ if and only if there is a positive constant $K$, independent on $A$ and $B$, such that $\frac{1}{K}\, A \leq B \leq K\, A$. To show Theorem \ref{teo1}, we establish, in Section 2, a triangle inequality which reduce our problem to estimate the essential norm of $W_{u', \varphi}: \mathcal{B}^{v_{\log}} \to H_{v_{\log}}^{\infty}$. Such estimation is found in Section 3 in terms of the essential norm of $W_{u,\varphi}:H_{v_{\log}}^{\infty} \to H_{w_{\log}}^{\infty}$. Finally, in Section 4, we show Theorem \ref{teo1}. As a consequence of our results, in Section 5, we characterize continuity, compactness and we estimate the essential norm of the composition operator $C_\varphi$ acting on the logarithmic-Zygmund space.

\vspace{0.2cm}
We want finish this introduction by mentioning that throughout this paper,
constants are denoted by $C$ or $C_v$ (if depending only on $v$), they are positive and
may differ from one occurrence to the other.

\section{A triangle inequality for the essential norm}
Let  $\mu_1$ and  $\mu_2$ be two weights defined on  $\Bbb D$. In this section we find upper bound for the essential norm of the operator  $W_{u, \varphi}:\mathcal{B}^{\mu_1}\to \mathcal{B}^{\mu_2}$ in terms of the essential norm of the operators  $W_{u', \varphi}:\mathcal{B}^{\mu_1}\to H_{\mu_{2}}^{\infty}$ and  $W_{u\varphi', \varphi}:H_{\mu_{1}}^{\infty}\to H_{\mu_{2}}^{\infty}$. To this end, for a weight  $\mu$ we set the class
$$
\widetilde{\mathcal{B}}^{\mu}=\{f \in  \mathcal{B}^{\mu}: f(0)=0\},
$$
which is a closed subspace of  $\mathcal{B}^{\mu}$. With this notation, we have the following result:

\begin{lem}\label{lem2}
If the operator   $W_{u, \varphi}:\widetilde{\mathcal{B}}^{\mu_1}\to \mathcal{B}^{\mu_2}$  is continuous, then
\begin{equation}\label{ec:2}
\|W_{u, \varphi} \|_{e}^{\widetilde{\mathcal{B}}^{\mu_1}\to \mathcal{B}^{\mu_2}} \leq \|W_{u', \varphi} \|_{e}^{\widetilde{\mathcal{B}}^{\mu_1}\to H_{\mu_{2}}^{\infty}}+\| W_{u\varphi', \varphi}\|_{e}^{H_{\mu_{1}}^{\infty}\to H_{\mu_{2}}^{\infty}}.
\end{equation}
\end{lem}

\begin{proof}
  Let us suppose that the operator $W_{u, \varphi}:\widetilde{\mathcal{B}}^{\mu_1}\to \mathcal{B}^{\mu_2}$ is continuous, then the continuous composition $D_{\mu_{2}}W_{u,\varphi}D_{\mu_1}^{-1}$ maps the space  $H_{\mu_1}^{\infty}$ into $H_{\mu_2}^{\infty}$, where $D_{\mu_{1}}:\widetilde{\mathcal{B}}^{\mu_1}\to H_{\mu_{1}}^{\infty}$  and   $D_{\mu_{2}}: \widetilde{\mathcal{B}}^{\mu_{2}}\to H_{\mu_{2}}^{\infty}$ denote the linear operators which transforms  each  $f \in H(\mathbb{D})$ into its derivative  $f'$. Clearly, the operators $D_{\mu_{1}}$ and    $D_{\mu_{2}}$ are isometry and therefore they are invertibles  with norms equal to 1. Furthermore, for every a $f\in H_{\mu_{1}}^{\infty}$ we have the relation
  $$
     D_{\mu_{2}}W_{u,\varphi}D_{\mu_1}^{-1}(f) = W_{u', \varphi}D_{\mu_1}^{-1}(f)+W_{u \varphi', \varphi}(f).
  $$
  From the above relation, we deduce an expression  for the operator $W_{u,\varphi}: \widetilde{\mathcal{B}}^{\mu_1}\to \mathcal{B}^{\mu_2}$ in terms of  $D_{\mu_{1}}$ and  $D_{\mu_{2}}$, that is,
\begin{eqnarray}\label{ec:1}
  W_{u,\varphi}: \widetilde{\mathcal{B}}^{\mu_1} &\to& \mathcal{B}^{\mu_2} \nonumber \\
  W_{u,\varphi}(f) &=& D_{\mu_2}^{-1}W_{u',\varphi}(f) +D_{\mu_2}^{-1}W_{u \varphi', \varphi}D_{\mu_1}(f).
\end{eqnarray}
Now, by definition of essential norm, given  $\epsilon>0$, we can find compact operators $\mathcal{K}_1:\widetilde{\mathcal{B}}^{\mu_1}\to H_{\mu_{2}}^{\infty}$ and
$\mathcal{K}_2: H_{\mu_{1}}^{\infty}\to H_{\mu_{2}}^{\infty}$ such that

\vspace{0.2cm}
\noindent $\displaystyle{
\|W_{u', \varphi} \|_{e}^{\widetilde{\mathcal{B}}^{\mu_{1}}\to H_{\mu_{2}}^{\infty}}+\| W_{ u \varphi', \varphi}\|_{e}^{H_{\mu_{1}}^{\infty}\to H_{\mu_{2}}^{\infty}}}
$
$$
 \geq \frac{1}{1+\epsilon}\left(\|W_{u', \varphi}-\mathcal{K}_1 \|^{\widetilde{\mathcal{B}}^{\mu_{1}}\to H_{\mu_{2}}^{\infty}}
+\| W_{u \varphi', \varphi}-\mathcal{K}_2\|^{H_{\mu_{1}}^{\infty}\to H_{\mu_{2}}^{\infty}}\right).
$$
Thus, since the operators  $D_{\mu_{1}}$ and  $D_{\mu_{2}}$ are isometries, we can write

\vspace{0.2cm}
\noindent $\displaystyle{
\|W_{u', \varphi} \|_{e}^{\widetilde{\mathcal{B}}^{\mu_{1}}\to H_{\mu_{2}}^{\infty}}+\| W_{ u \varphi', \varphi}\|_{e}^{H_{\mu_{1}}^{\infty}\to H_{\mu_{2}}^{\infty}}}
$
\begin{eqnarray*}
 &\geq&\frac{1}{1+\epsilon}\left( \|D_{\mu_2}^{-1}W_{u', \varphi}-D_{\mu_2}^{-1}\mathcal{K}_1 \|^{\widetilde{\mathcal{B}}^{\mu_1}\to \mathcal{B}^{\mu_2}}+
\|D_{\mu_2}^{-1} W_{u \varphi', \varphi} -D_{\mu_2}^{-1}\mathcal{K}_2\|^{H_{\mu_1}^{\infty}\to \mathcal{B}^{\mu_2}}   \right)\\
&\geq& \frac{1}{1+\epsilon}\left( \|D_{\mu_2}^{-1}W_{u', \varphi}-D_{\mu_2}^{-1}\mathcal{K}_1 \|^{\widetilde{\mathcal{B}}^{\mu_1}\to \mathcal{B}^{\mu_2}}+
\|D_{\mu_2}^{-1} W_{u \varphi', \varphi} D_{\mu_1} -D_{\mu_2}^{-1}\mathcal{K}_2D_{\mu_1}\|^{\widetilde{\mathcal{B}}^{\mu_1}\to \mathcal{B}^{\mu_2}}   \right)\\
&\geq& \frac{1}{1+\epsilon}\|D_{\mu_2}^{-1}W_{u', \varphi}+D_{\mu_2}^{-1} W_{u \varphi', \varphi}D_{\mu_1}-(D_{\mu_2}^{-1}\mathcal{K}_1+
D_{\mu_2}^{-1}\mathcal{K}_2D_{\mu_1}) \|^{\widetilde{\mathcal{B}}^{\mu_1}\to \mathcal{B}^{\mu_2}}\\
&=&\frac{1}{1+\epsilon}\|W_{u, \varphi}-\mathcal{K} \|^{\widetilde{\mathcal{B}}^{\mu_1}\to \mathcal{B}^{\mu_2}}\geq \frac{1}{1+\epsilon}\|W_{u,\varphi} \|_{e}^{\widetilde{\mathcal{B}}^{\mu_1}\to \mathcal{B}^{\mu_2}},
\end{eqnarray*}
where we have used  (\ref{ec:1}) in the last equality and the fact that the operator $\mathcal{K}=D_{\mu_2}^{-1}\mathcal{K}_1+D_{\mu_2}^{-1}\mathcal{K}_2D_{\mu_1}:\widetilde{\mathcal{B}}^{\mu_1}\to \mathcal{B}^{\mu_2}$ is compact since the composition of compact and continuous operators is a compact operator. The result follows because  $\epsilon>0$ was arbitrary.
\end{proof}

 Now, we show that the essential norm of the weighted composition operator, acting on Bloch-type spaces, does not change if we restrict the domain to $\widetilde{\mathcal{B}}^\mu$.

\begin{lem}\label{le-es-rest}
If  $ W_{u, \varphi}:\mathcal{B}^{\mu_1}\to \mathcal{B}^{\mu_2}$ is continuous, then
$$
\| W_{u, \varphi} \|_{e}^{\mathcal{B}^{\mu_1}\to \mathcal{B}^{\mu_2}}=\| W_{u, \varphi} \|_{e}^{\mathcal{\widetilde{B}}^{\mu_1}\to \mathcal{B}^{\mu_2}}.
$$
\end{lem}

\begin{proof}
This proof uses similar arguments given in \cite{esmaeili-lindstrom}, (see Lemma 3.1). Since  $\widetilde{\mathcal{B}}^{\mu_1}\subseteq \mathcal{B}^{\mu_1}$, it is clear that all compact operator from $\mathcal{B}^{\mu_1}$ into $\mathcal{B}^{\mu_2}$ is also a compact operator from $\widetilde{\mathcal{B}}^{\mu_1}$ into $\mathcal{B}^{\mu_2}$; hence we have
$$
\|W_{u, \varphi} \|_{e}^{\mathcal{B}^{\mu_1}\to \mathcal{B}^{\mu_2}}\ge\|W_{u, \varphi} \|_{e}^{\mathcal{\widetilde{B}}^{\mu_1}\to \mathcal{B}^{\mu_2}}.
$$
To show the reverse inequality, we observe that if   $T: \mathcal{B}^{\mu_1}\to \mathcal{B}^{\mu_2}$ is any compact operator, then  we can write

\vspace{0.2cm}\noindent
$\displaystyle{
\| W_{u, \varphi}-T\|^{\mathcal{B}^{\mu_1}\to \mathcal{B}^{\mu_2}}=
\sup_{\|f \|_{\mathcal{B}^{\mu_1}\leq 1}}
\| W_{u, \varphi}(f)-T(f)\|_{\mathcal{B}^{\mu_2}}}
$
\begin{eqnarray*}
&=&\sup_{\|f \|_{\mathcal{B}^{\mu_1}\leq 1}} \| W_{u, \varphi}(f)+f(0)W_{u, \varphi}(\mathbf{1})-f(0)W_{u,\varphi}(\mathbf{1})-T(f-f(0)\mathbf{1}+f(0)\mathbf{1})\|_{\mathcal{B}^{\mu_2}}\\
&\leq& \sup_{\|f \|_{\mathcal{B}^{\mu_1}\leq 1}} \| W_{u, \varphi}(f-f(0)\mathbf{1}) -T|_{\widetilde{\mathcal{B}}^{\mu_1}}(f-f(0)\mathbf{1})\|_{\mathcal{B}^{\mu_2}}
+\sup_{\|f \|_{\mathcal{B}^{\mu_1}\leq 1}}\|W_{u, \varphi}(f(0)\mathbf{1})-T(f(0)\mathbf{1}) \|_{\mathcal{B}^{\mu_2}}\\
&\leq& \sup_{g \in \widetilde{\mathcal{B}}^{\mu_1},\,\|g \|_{\mathcal{B}^{\mu_1}\leq 1} }\|W_{u, \varphi}(g) -T|_{\widetilde{\mathcal{B}}^{\mu_1}}(g)\|_{\mathcal{B}^{\mu_2}}
+ \sup_{h \in \mathcal{A},\,\|h \|_{\mathcal{B}^{\mu_1}}\leq 1}\|W_{u, \varphi}(h)-T|_{\mathcal{A}}(h) \|_{\mathcal{B}^{\mu_2}}\\
&=&\| W_{u, \varphi}-P_1\|^{\widetilde{\mathcal{B}}^{\mu_1}\to \mathcal{B}^{\mu_2}}+\| W_{u, \varphi}-Q_1\|^{\mathcal{A}\to \mathcal{B}^{\mu_2}},
\end{eqnarray*}
where  $\mathbf{1}(z)=1$ for all $z\in\Bbb D$, $\mathcal{A}$ denotes the space of all constant functions in  $\mathcal{B}^{\mu_1}$, $P_1=T|_{\widetilde{\mathcal{B}}^{\mu_1}}$ and $Q_1=T|_{\mathcal{A}}$.

\vspace{0.2cm}
On the other hand, by definition of essential norm, given $\epsilon>0$, we can find compact operators   $P: \widetilde{\mathcal{B}}^{\mu_1} \to \mathcal{B}^{\mu_2}$ and $Q:\mathcal{A} \to \mathcal{B}^{\mu_2}$ such that
$$
\left\|W_{u, \varphi}-P \right\|^{\widetilde{\mathcal{B}}^{\mu_1} \to \mathcal{B}^{\mu_2}}+\left\|W_{u, \varphi}-Q \right\|^{\mathcal{A} \to \mathcal{B}^{\mu_2}}
\leq (1+\epsilon)\left(\|W_{u, \varphi} \|_{e}^{\widetilde{\mathcal{B}}^{\mu_{1}} \to \mathcal{B}^{\mu_{2}}}
+\|W_{u, \varphi} \|_{e}^{\mathcal{A} \to \mathcal{B}^{\mu_{2}}}\right).
$$
But each  $f\in \mathcal{B}^{\mu_{1}}$ can be written as $f=f(0)\cdot\mathbf{1} + g$, where  $g\in \widetilde{\mathcal{B}}^{\mu_{1}}$.  Furthermore, we can see that if  $f\in \mathcal{A}$ then  $g$ is the null function and if  $f\in \widetilde{\mathcal{B}}^{\mu_{1}}$ then  $f=g$. Thus, we can define the operator  $T: \mathcal{B}^{\mu_{1}}\to \mathcal{B}^{\mu_{2}}$ by
$$
 Tf = f(0) Q\left(\mathbf{1}\right) + P(g),\hspace{0.2cm}(f=f(0)\cdot \mathbf{1} + g\in \mathcal{B}^{\mu_{1}}).
$$
Clearly,  $T$ is linear and compact operator from $\mathcal{B}^{\mu_{1}}$ into $\mathcal{B}^{\mu_{2}}$. Indeed, if  $\left\{f_k\right\}$ is a bounded sequence in  $\mathcal{B}^{\mu_1}$, then there exists a bounded sequence $\left\{g_k\right\}$ in $\widetilde{\mathcal{B}}^{\mu_{1}}$ such that
$$f_k(z)=f_k(0)+g_k(z)$$
for all  $z \in \mathbb{D}$.  Bolzano-Weierstrass's theorem tell us that the numerical sequence  $\left\{f_k(0)\right\}$ has a convergent subsequence and since the operator  $P: \widetilde{\mathcal{B}}^{\mu_1} \to \mathcal{B}^{\mu_2}$ is compact, the sequence  $\left\{P\left(g_k\right)\right\}$ also has a convergent subsequence in   $\mathcal{B}^{\mu_2}$. Hence  $\left\{T\left(f_k\right)\right\}$  has a convergent subsequence in   $\mathcal{B}^{\mu_2}$ and $T: \mathcal{B}^{\mu_{1}}\to \mathcal{B}^{\mu_{2}}$ is compact as was claimed.

\vspace{0.2cm}
Now, since  $T|_{\widetilde{\mathcal{B}}^{\mu_{1}}}= P$ and $T|_{\mathcal{A}}=Q$. We obtain that
\begin{eqnarray*}
\|W_{u, \varphi} \|_{e}^{\mathcal{B}^{\mu_{1}}\to \mathcal{B}^{\mu_{2}}}
&\leq& \| W_{u, \varphi}-T\|^{\mathcal{B}^{\mu_{1}}\to \mathcal{B}^{\mu_{2}}}\\
&\leq&\|W_{u, \varphi}-P \|^{\widetilde{\mathcal{B}}^{\mu_{1}} \to \mathcal{B}^{\mu_{2}}}+\|W_{u, \varphi}-Q \|^{\mathcal{A} \to \mathcal{B}^{\mu_{2}}}\\
&\leq& (1+\epsilon)\left(\|W_{u, \varphi} \|_{e}^{\widetilde{\mathcal{B}}^{\mu_{1}} \to \mathcal{B}^{\mu_{2}}}
+\|W_{u, \varphi} \|_{e}^{\mathcal{A} \to \mathcal{B}^{\mu_{2}}}\right).
\end{eqnarray*}
and since  $\epsilon>0$ was arbitrary, we conclude that
$$
\|W_{u, \varphi} \|_{e}^{\mathcal{B}^{\mu_{1}}\to \mathcal{B}^{\mu_{2}}}\leq \|W_{u, \varphi} \|_{e}^{\widetilde{\mathcal{B}}^{\mu_{1}} \to \mathcal{B}^{\mu_{2}}}+\|W_{u, \varphi} \|_{e}^{\mathcal{A} \to \mathcal{B}^{\mu_{2}}}
$$
and the result follows since the operator  $W_{u, \varphi}: \mathcal{A} \to \mathcal{B}^{\mu_{2}}$ is compact and therefore its essential norm is zero.
\end{proof}

As a consequence of Lemmas \ref{lem2} and  \ref{le-es-rest} we have the following result:
\begin{thm}\label{th-cota-sup-norm-es}
If the operator   $W_{u, \varphi}:\widetilde{\mathcal{B}}^{\mu_1}\to \mathcal{B}^{\mu_2}$ is continuous, then
\begin{equation}\label{ec:2}
\|W_{u, \varphi} \|_{e}^{\mathcal{B}^{\mu_1}\to \mathcal{B}^{\mu_2}} \leq \|W_{u', \varphi} \|_{e}^{\widetilde{\mathcal{B}}^{\mu_1}\to H_{\mu_{2}}^{\infty}}+\| W_{u\varphi', \varphi}\|_{e}^{H_{\mu_{1}}^{\infty}\to H_{\mu_{2}}^{\infty}}.
\end{equation}
\end{thm}

\section{Estimation of the essential norm of  $W_{u', \varphi}: \mathcal{B}^{v_{\log}} \to H_{v_{\log}}^{\infty}$}
From the conclusion of Theorem \ref{th-cota-sup-norm-es}, we see that  we have to estimate the essential norm of the operators $W_{u\varphi', \varphi}:H_{\mu_{1}}^{\infty}\to H_{\mu_{2}}^{\infty}$ and $W_{u', \varphi}: \widetilde{\mathcal{B}}^{\mu_1}\to H_{\mu_{2}}^{\infty}$. However, in the case of the weights $v_{\log}$ and $w_{\log}$, the first one can be estimate using the results of Malav\'e-Ram{\'i}rez and Ramos-Fern\'andez in \cite{malave-ramos}, since theses weights are radial and typical. Hence, in this section, we look for an upper bound for  $\|W_{u', \varphi} \|_{e}^{\widetilde{\mathcal{B}}^{v_{\log}}\to H_{v_{\log}}^{\infty}}$. To this end, recall that for all  $f \in \mathcal{B}^{v_{\log}}$ the following relation holds:
\begin{equation}\label{des-func-eval}
|f(z)|\leq \left[1+\log \log \left(\dfrac{2}{1-|z|}\right) -\log \log(2)\right]\|f \|_{{\mathcal{B}}^{v_{\log}}}.
\end{equation}
This fact allow us to show the following result, where $w_{\log}$ is the weight defined in (\ref{def-peso-w}).

\begin{thm}\label{lem3}
Suppose that  $u:\mathbb{D}\to \mathbb{C}$ and  $\varphi:\mathbb{D} \to \mathbb{D}$ are holomorphic functions and that the operator
$W_{u, \varphi}:\mathcal{B}^{v_{\log}} \to H_{v_{\log}}^{\infty}$ is continuous. Then there exists a constant  $C>0$ such that
$$
\|W_{u, \varphi} \|_{e}^{\mathcal{B}^{v_{\log}}\to H_{v_{\log}}^{\infty}}\leq C \lim_{r \to 1^{-}}\sup_{|\varphi(z)|>r}\dfrac{v_{\log}(z)}{w_{\log}(\varphi(z))}|u(z)|.
$$
\end{thm}
\begin{proof}
For each  $r \in (0,1)$, we consider $K_r: \mathcal{B}^{v_{\log}}\to \mathcal{B}^{v_{\log}}$ defined by  $K_r(f)=f_r$,
where  $f_r$ is the dilatation of  $f$ given by  $f_r(z)=f(rz)$. It is known that  (see \cite{CCFR13}) for each  $r\in (0,1)$ the operator $K_r$ is continuous and compact on $\mathcal{B}^{v_{\log}}$.

\vspace{0.2cm}
 Let $\left\{r_n\right\}$ any sequence in $(0,1)$  and consider the compact operators
$K_n: \mathcal{B}^{v_{\log}} \to \mathcal{B}^{v_{\log}}$ given by  $K_n=K_{r_n}$, then the operators  $ W_{u, \varphi}K_n: \mathcal{B}^{v_{\log}}\to H_{v_{\log}}^{\infty}$ also are compacts for every  $n \in \mathbb{N}$, since  $W_{u, \varphi}:\mathcal{B}^{v_{\log}}\to H_{v_{\log}}^{\infty}$ is a continuous operator. Hence, by definition of essential norm we can write
$$
\| W_{u, \varphi} \|_{e}^{\mathcal{B}^{v_{\log}} \to H_{v_{\log}}^{\infty}} \leq \limsup_{n \to \infty}\|W_{u, \varphi}-W_{u, \varphi}K_n \|^{\mathcal{B}^{v_{\log}}\to H_{v_{\log}}^{\infty}}.
$$
Furthermore, for any  $f \in \mathcal{B}^{v_{\log}}$ such that  $\|f \|_{\mathcal{B}^{v_{\log}}}\leq 1$, we have
\begin{eqnarray*}
\left\|\left(W_{u, \varphi}-W_{u, \varphi}K_n\right)(f) \right\|_{H_{v_{\log}}^{\infty}}
 &=& \|u(f-f_{r_n})\circ \varphi\|_{H_{v_{\log}}^{\infty}}\\
 &=&\sup_{z \in \mathbb{D}}v_{\log}(z)|u(z)f(\varphi(z))-u(z)f(r_n \varphi(z))|\\
 &=&\sup_{z \in \mathbb{D}}v_{\log}(z)|f(\varphi(z))-f(r_n\varphi(z))||u(z)|.
\end{eqnarray*}
Now, we fix  $N \in \mathbb{N}$ and consider, for  $z\in\Bbb D$, the cases  $|\varphi(z)|\leq r_N$ and  $|\varphi(z)|>r_N$.

\vspace{0.2cm}
\noindent {\bf Case 1: $|\varphi(z)|\leq r_N$}

\noindent Since the operator  $W_{u, \varphi}: \mathcal{B}^{v_{\log}}\to H_{v_{\log}}^{\infty}$ is continuous, then $u=W_{u, \varphi}\left(\mathbf{1}\right)\in H_{v_{\log}}^{\infty}$. Hence
\begin{eqnarray*}
\sup_{|\varphi(z)|\leq r_N}v_{\log}(z)|f(\varphi(z))-f(r_n\varphi(z))||u(z)|
&\leq& \| u\|_{H_{v_{\log}}^{\infty}}\sup_{|\varphi(z)|
\leq r_N}|f(\varphi(z))-f(r_n\varphi(z))|\\
&\leq&\| u\|_{H_{v_{\log}}^{\infty}}\sup_{|w|\leq r_N}|f(w)-f_{r_n}(w)|\to 0,
\end{eqnarray*}
as $n\to\infty$, since  is a known fact that  $f_r\to f$ uniformly on compact subsets of $\Bbb D$ as $r\to 1^{-}$.

\vspace{0.2cm}
\noindent {\bf Case 2: $|\varphi(z)|> r_N$}\\
By triangle inequality, we have

\vspace{0.2cm}
\noindent $\displaystyle{
\sup_{|\varphi(z)|> r_N}v_{\log}(z)|f(\varphi(z))-f(r_n\varphi(z))||u(z)|
}$
$$
\leq \sup_{|\varphi(z)|> r_N}v_{\log}(z)|f(\varphi(z))||u(z)|+ \sup_{|\varphi(z)|> r_N}v_{\log}(z)|f(r_n\varphi(z))||u(z)|.
$$
Thus, it is enough to find upper bounds for the expression in the right side of the above inequality. Put
\begin{eqnarray*}
  L_1 &=& \sup_{|\varphi(z)|> r_N}v_{\log}(z)|f(\varphi(z))||u'(z)|\\
  L_2 &=& \sup_{|\varphi(z)|> r_N}v_{\log}(z)|f(r_n\varphi(z))||u'(z)|.
\end{eqnarray*}
By the inequality (\ref{des-func-eval}), the fact that  $\|f\|_{\mathcal{B}^{v_{\log}}}\leq 1$ and since   $\log \log \left(\frac{2}{1-|\varphi(z)|} \right)\leq \log \log \left(\frac{4}{1-|\varphi(z)|^2} \right)$, we obtain
\begin{eqnarray*}
L_1 &\leq& \sup_{|\varphi(z)|> r_N}v_{\log}(z)|u(z)|\left[ 1+\log \log \left(\frac{2}{1-|\varphi(z)|}\right) -\log \log(2)\right]\\
&\leq& \sup_{|\varphi(z)|> r_N}\dfrac{v_{\log}(z)}{w_{\log}(\varphi(z))}|u(z)|
\left[  \dfrac{1+ \log \log \left(\frac{2}{1-|\varphi(z)|} \right)-\log\log (2)}{
\log \log \left(\frac{4}{1-|\varphi(z)|^2} \right)} \right]\\
&\leq& \sup_{|\varphi(z)|> r_N}\dfrac{v_{\log}(z)}{w_{\log}(\varphi(z))}|u(z)|\left[w_{\log}(\varphi(z)) -\log\log(2) w_{\log}(\varphi(z))  +1  \right]\\
&\leq& \sup_{|\varphi(z)|> r_N}\dfrac{v_{\log}(z)}{w_{\log}(\varphi(z))}|u(z)|\left[w_{\log}(0) -\log\log(2) w_{\log}(0)  +1  \right]\\
&\leq& \sup_{|\varphi(z)|> r_N}\dfrac{v_{\log}(z)}{w_{\log}(\varphi(z))}|u(z)|\left[2w_{\log}(0) +1  \right]\\
&\leq& C\sup_{|\varphi(z)|> r_N}\dfrac{v_{\log}(z)}{w_{\log}(\varphi(z))}|u(z)|,
\end{eqnarray*}
where  $C= 2w_{\log}(0) +1  >0$ and we have used that  $w_{\log}(r)$ is decreasing on  $[0,1]$.

\vspace{0.2cm}
In a similar way, we have that
$$
L_2
\leq \sup_{|\varphi(z)|> r_N}\dfrac{v_{\log}(z)}{w_{\log}(r_n\varphi(z))}|u(z)|\left[ 2w_{\log}(0) +1  \right]
\leq \sup_{|\varphi(z)|> r_N}\dfrac{v_{\log}(z)}{w_{\log}(\varphi(z))}|u(z)|\left[ 2w_{\log}(0) +1  \right].
$$
Hence, we can say that there exists a constant  $C>0$ such that
$$
\left\|\left(W_{u, \varphi}-W_{u, \varphi}K_n\right)(f) \right\|_{H_{v_{\log}}^{\infty}}\leq C \sup_{|\varphi(z)|> r_N}\dfrac{v_{\log}(z)}{w_{\log}(\varphi(z))}|u(z)|.
$$
Therefore, taking  $N \to \infty$  we have  $r_{N}\to 1^{-}$,
$$
\left\|W_{u,\varphi}\right\|_{e}^{\mathcal{B}^{v_{\log}}\to H_{v_{\log}}^{\infty}}\leq  C\lim_{r_N \to 1^{-}}\sup_{|\varphi(z)|> r_N}\dfrac{v_{\log}(z)}{w_{\log}(\varphi(z))}|u(z)|.
$$
This shows the result.
\end{proof}

As a consequence of the above result  we have the following estimation:
\begin{cor}\label{cor-3-2}
 Suppose that  $u:\mathbb{D}\to \mathbb{C}$ and  $\varphi:\mathbb{D} \to \mathbb{D}$ are holomorphic functions and that the operator
$W_{u, \varphi}:\mathcal{B}^{v_{\log}} \to H_{v_{\log}}^{\infty}$ is continuous. Then there exists a constant  $C>0$  such that
$$
\|W_{u, \varphi} \|_{e}^{\mathcal{B}^{v_{\log}}\to H_{v_{\log}}^{\infty}}\leq C \limsup_{n \to \infty}\dfrac{\|u \varphi^n \|_{H_{v_{\log}}^{\infty}}}{\|g_n \|_{H_{w_{\log}}^{\infty}}},
$$
where  $g_n$ are the functions defined in Theorem \ref{cor2}.
\end{cor}

\begin{proof}
 Clearly, the weight $w_{\log}$ is radial, typical and decreasing on $(0,1)$ and hence (see \cite{HKLRS12} or \cite{malave-ramos}), this weight is essential, that is, $w_{\log}\simeq \widetilde{w}_{\log}$, where, for a weight $v$, $\widetilde{v}$ denotes its {\it associated weight} given by
 $$
  \widetilde{v}(z)= \left(\sup_{\|f\|_{H_v^\infty}\leq 1}\left|f(z)\right|\right)^{-1}
 $$
 with $z\in\Bbb D$. Also (see \cite{CMR14}), it is easy to see that $v_{\log}\simeq v_3$, where $v_3$ is given by
 \begin{equation}\label{peso}
  v_{3}(z)=(1-|z|)\log\left(\frac{3}{1-|z|} \right)
 \end{equation}
 with $z\in\Bbb D$. Thus,  $H^\infty_{v_{\log}}$ is equal to $H^\infty_{v_3}$ with norms equivalents. Hence, from Theorem \ref{lem3} we can write
 \begin{eqnarray*}
  \left\|W_{u,\varphi}\right\|_{e}^{\mathcal{B}^{v_{\log}}\to H_{v_{\log}}^{\infty}}
  &\leq&  C\lim_{r \to 1^{-}}\sup_{|\varphi(z)|> r}\dfrac{v_{\log}(z)}{w_{\log}(\varphi(z))}|u(z)|\\
  &\leq&  C\lim_{r \to 1^{-}}\sup_{|\varphi(z)|> r}\dfrac{v_{3}(z)}{\widetilde{w}_{\log}(\varphi(z))}|u(z)|\\
  &=& C \left\|W_{u,\varphi}\right\|_{e}^{ H^\infty_{w_{\log}}\to H_{v_{3}}^{\infty}}\\
  &\leq&C \limsup_{n \to \infty}\dfrac{\|u \varphi^n \|_{H_{v_{\log}}^{\infty}}}{\|g_n \|_{H_{w_{\log}}^{\infty}}},
 \end{eqnarray*}
 where we have used  Theorem 2.1 in \cite{montes} in the equality and the last inequality is due to Theorem 2.4 in \cite{HKLRS12} (see also \cite{malave-ramos}, Theorem 4.3).
\end{proof}

Now, using the definition of the functions $g_n$,  the definition of the functional $J_{u}$, given in Theorem \ref{cor2}, and the fact that
$\left\|f\right\|_{\widetilde{\mathcal{B}}^{v}} = \left\|f'\right\|_{H_v^\infty}$ we can conclude:

\begin{cor}\label{cor-cota-sup-nom-es-2}
Suppose that  $u:\mathbb{D}\to \mathbb{C}$ and  $\varphi:\mathbb{D} \to \mathbb{D}$ are holomorphic functions and  that the operator
$W_{u, \varphi}:\mathcal{B}^{v_{\log}} \to H_{v_{\log}}^{\infty}$ is continuous. Then there exists a constant  $C>0$  such that
$$
\|W_{u, \varphi} \|_{e}^{\mathcal{B}^{v_{\log}}\to H_{v_{\log}}^{\infty}}\leq C
\limsup_{n \to \infty}\dfrac{(n+1)\| J_{u}(\varphi^n)\|_{\mathcal{B}^{v_{\log}}}}{\|g_{n+1}\|_{\mathcal{B}^{w_{\log}}}}.
$$
\end{cor}

\section{The essential norm of $W_{u, \varphi}: \mathcal{B}^{v_{\log}}\to \mathcal{B}^{v_{\log}}$. Proof of Theorem \ref{teo1}}
Now we can give the proof of Theorem \ref{teo1}. From Theorem \ref{th-cota-sup-norm-es} we have
$$
\| W_{u, \varphi} \|_{e}^{\mathcal{B}^{v_{\log}}\to \mathcal{B}^{v_{\log}}}\leq \| W_{u', \varphi}\|_{e}^{\mathcal{B}^{v_{\log}}\to H_{v_{\log}}^{\infty}}+\| W_{u \varphi', \varphi} \|_{e}^{H_{v_{\log}}^{\infty} \to H_{v_{\log}}^{\infty}}.
$$
Thus, by Corollary \ref{cor-cota-sup-nom-es-2} and Theorem 4.4 in \cite{malave-ramos}, we can find a constant  $C>0$ such that

\vspace{0.2cm}\noindent $\displaystyle{
\| W_{u', \varphi}\|_{e}^{\mathcal{B}^{v_{\log}}\to H_{v_{\log}}^{\infty}}+\| W_{u \varphi',\varphi} \|_{e}^{H_{v_{\log}}^{\infty} \to H_{v_{\log}}^{\infty}}}$
\begin{eqnarray*}
&\leq& C\left(\limsup_{n \to \infty}\dfrac{(n+1)\| J_{u}(\varphi^n)\|_{\mathcal{B}^{v_{\log}}}}{\|g_{n+1}\|_{\mathcal{B}^{w_{\log}}}}+ \limsup_{n \to \infty} \dfrac{\|I_{u}(\varphi^n)\|_{\mathcal{B}^{v_{\log}}}}{\|g_n\|_{\mathcal{B}^{v_{\log}}}}\right)\\
&\leq& C\max \left\{\limsup_{n \to \infty}\dfrac{(n+1)\| J_{u}(\varphi^n)\|_{\mathcal{B}^{v_{\log}}}}{\|g_{n+1}\|_{\mathcal{B}^{w_{\log}}}}, \limsup_{n \to \infty} \dfrac{\|I_{u}(\varphi^n)\|_{\mathcal{B}^{v_{\log}}}}{\|g_n\|_{\mathcal{B}^{v_{\log}}}} \right\}.
\end{eqnarray*}

Now we go to give a lower bound for $\| W_{u, \varphi} \|_{e}^{\mathcal{B}^{v_{\log}}\to \mathcal{B}^{v_{\log}}}$. To this end, let  $K:\mathcal{B}^{v_{\log}}\to \mathcal{B}^{v_{\log}}$ be a compact operator and let  $\{ z_n\}_{n \in \mathbb{N}}$ be a sequence in  $\mathbb{D}$ such that  $|\varphi(z_n)|\to 1^-$ as $n \to \infty$. The following sequence was defined in  \cite{ye},
$$
f_n(z)=\frac{3}{a_n}\left[ \log \log\left( \frac{4}{1-\overline{\varphi(z_n)}z}\right)\right]^2-\frac{2}{a_n^2}\left[\log \log \left(\frac{4}{1-\overline{\varphi(z_n)}z}\right) \right]^3,
$$
where  $a_n=\log \log \left(\dfrac{4}{1-|\varphi(z_n)|^2} \right)$. In \cite{ye} the author shown that  $\{f_n \}$ is a bounded sequence in  $\mathcal{B}^{v_{\log}}$, that is, there exists a constant $M>0$ such that  $\left\|f_n\right\|_{\mathcal{B}^{v_{\log}}}\leq M$ for all $n \in \mathbb{N}$. This sequence converges to zero uniformly on compact subsets of  $\mathbb{D}$. The derivatives of $f_n$ is given by
 $$
 f'_n(z)=\dfrac{\frac{6 \overline{\varphi(z_n)}}{a_n}\log\log \left(\frac{4}{1-\overline{\varphi(z_n)}z} \right)}{(1-\overline{\varphi(z_n)}z)\log \left( \frac{4}{1-\overline{\varphi(z_n)}z}\right)}-\dfrac{\frac{6 \overline{\varphi(z_n)}}{a_n^2}\left[\log \log \left( \frac{4}{1-\overline{\varphi(z_n)}z}\right)\right]^2}{(1-\overline{\varphi(z_n)}z)\log \left( \frac{4}{1-\overline{\varphi(z_n)}z}\right)}.
 $$
Hence, we have  $f_n'(\varphi(z_n))=0$,   $f_n(\varphi(z_n))=a_n$ and
\begin{eqnarray*}
 M \|W_{u, \varphi}-K \|^{\mathcal{B}^{v_{\log}}\to \mathcal{B}^{v_{\log}}} &\ge& \limsup_{n \to \infty}\| \left(W_{u, \varphi}-K\right)(f_n) \|_{\mathcal{B}^{v_{\log}}}\\
 &\geq &\limsup_{n \to \infty}\| W_{u, \varphi}(f_n) \|_{\mathcal{B}^{v_{\log}}}-\limsup_{n \to \infty}\|K f_n \|_{\mathcal{B}^{v_{\log}}} \\
  &=&  \limsup_{n \to \infty}\|W_{u, \varphi}(f_n) \|_{\mathcal{B}^{v_{\log}}},
\end{eqnarray*}
where we have used the known fact (see \cite{tjani}) that if $K:\mathcal{B}^{v_{\log}}\to \mathcal{B}^{v_{\log}}$ is compact then
$$\lim_{n \to \infty}\|K f_n \|_{\mathcal{B}^{v_{\log}}}=0$$
for all bounded sequence $\left\{f_n\right\}\subset \mathcal{B}^{v_{\log}}$ converging to zero uniformly on compact subsets of $\Bbb D$.

 On the other hand,
\begin{eqnarray*}
\|  W_{u, \varphi}(f_n)\|_{\mathcal{B}^{v_{\log}}}&=&\|u f_n(\varphi) \|_{\mathcal{B}^{v_{\log}}}\\
&\geq&\limsup_{n \to \infty}v_{\log}\left(z_n\right)|u'(z_n)f_n(\varphi(z_n))+u(z_n)\varphi'(z_n)f_n'(\varphi(z_n))|\\
&=& \limsup_{n \to \infty}v_{\log}\left(z_n\right)|u'(z_n)f_n(\varphi(z_n))|\\
&=&\limsup_{n \to \infty}v_{\log}\left(z_n\right) |u'(z_n)|\log\log \left(\dfrac{4}{1-|\varphi(z_n)|^2} \right)\\
&=&\limsup_{n \to \infty} \dfrac{v_{\log}(z_n)}{w_{\log}(\varphi(z_n))}|u'(z_n)|\ge C \limsup_{n \to \infty}\dfrac{v_{\log}(z_n)}{\widetilde{w}_{\log}\left(\varphi(z_n)\right)}|u'(z_n)|,
\end{eqnarray*}
where, in the last inequality, we have used that the weight $w_{\log}$ is essential.  Thus, since the sequence  $\left\{z_n\right\}$ such that  $\left|\varphi\left(z_n\right)\right|\to 1^-$ was arbitrary, we can deduce that
\begin{eqnarray}\label{cota-inf-1}
  \| W_{u, \varphi} \|_{e}^{\mathcal{B}^{v_{\log}}\to \mathcal{B}^{v_{\log}}} &\ge& C \limsup_{|\varphi(z)|\to 1^-} \dfrac{v_{\log}(z)}{\widetilde{w}_{\log}(\varphi(z))}|u'(z)|\\
   &\geq& C\, \| W_{u', \varphi} \|_{e}^{H_{w_{\log}}^{\infty}\to H_{v_{\log}}^{\infty}}\nonumber \\
   &=& C\, \limsup_{n \to \infty}\dfrac{\|u'\varphi^n \|_{H_{v_{\log}}^{\infty}}}{\|g_n\|_{H_{w_{\log}}^{\infty}}}\nonumber\\
    &=& C\, \limsup_{n \to \infty} \dfrac{(n+1)\|J_u(\varphi^n) \|_{\mathcal{B}^{v_{\log}}}}{\|g_{n+1} \|_{\mathcal{B}^{w_{\log}}}},\nonumber
\end{eqnarray}
where we have used the argument in the proof of Corollary  \ref{cor-3-2}.

\vspace{0.2cm}
 Now, we go to show that there exists a constant $C>0$ such that
  $$
   \| W_{u, \varphi} \|_{e}^{\mathcal{B}^{v_{\log}}\to \mathcal{B}^{v_{\log}}}\geq C\limsup_{n \to \infty} \dfrac{\|I_{u}(\varphi^n)\|_{\mathcal{B}^{v_{\log}}}}{\|g_n\|_{\mathcal{B}^{v_{\log}}}}.
  $$
 As before, we consider a sequence  $\left\{z_n\right\}\subset\Bbb D$ such that  $\left|\varphi\left(z_n\right)\right|\to 1^-$ and we define the functions
$$
h_n(z)=\frac{1}{\overline{\varphi(z_n)}a_n^2}\left[\log \log \left(\frac{4}{1-\overline{\varphi(z_n)}z} \right) \right]^3
-\frac{1}{\overline{\varphi(z_n)}a_n}\left[\log \log \left(\frac{4}{1-\overline{\varphi(z_n)}z} \right) \right]^2.
$$
Then, clearly $h_n\left(\varphi(z_n)\right)=0$ for all $n\in\Bbb N$, $h_n$ converges to zero uniformly on compact subsets of $\Bbb D$,
$$
h'_n(z)=\dfrac{\frac{3}{a_n^2}\left[\log \log \left( \frac{4}{1-\overline{\varphi(z_n)}z}\right)\right]^2}{(1-\overline{\varphi(z_n)}z)\log \left(\frac{4}{1-\overline{\varphi(z_n)}z} \right)}-\dfrac{\frac{2}{a_n}\log \log \left(\frac{4}{1-\overline{\varphi(z_n)}z} \right)}{(1-\overline{\varphi(z_n)}z)\log \left(\frac{4}{1-\overline{\varphi(z_n)}z} \right)},
$$
and hence
$$
h'_n(\varphi(z_n))=\left[(1-|\varphi(z_n)|^2)\log \left(\frac{4}{1-|\varphi(z_n)|^2} \right)\right]^{-1}.
$$
Furthermore, $\left\{h_n\right\}$ is a bounded sequence in $\mathcal{B}^{v_{\log}}$; that is, there exists a constant $M>0$ such that $\|h_n \|_{\mathcal{B}^{v_{\log}}}\leq M$ for all $n\in\Bbb N$. Indeed,
\begin{eqnarray}\label{ec:3}
  \|h_n \|_{\mathcal{B}^{v_{\log}}} &=& \sup_{z \in \mathbb{D}}v_{\log}(z)|h_n'(z)| \nonumber \\
   &\leq& \sup_{z \in \mathbb{D}}\frac{v_{\log}(z)\left[ \frac{3}{a_n^2} \left|\log \log \left(\frac{4}{(1-\overline{\varphi(z_n)}z)} \right) \right|^2+ \frac{2}{a_n}\left| \log \log \left(\frac{4}{(1-\overline{\varphi(z_n)}z)} \right) \right| \right]}{|1-\overline{\varphi(z_n)}z|\log \left(\frac{4}{|1-\overline{\varphi(z_n)}z|} \right)}
\end{eqnarray}
where we applied the triangle inequality and the fact that $|\log(w)|\ge \log(|w|)$ for each $w \in \mathbb{D}$. Furthermore, since
$$
\lim_{x \to 0^{+}}\frac{\sqrt{\log^2 \left( \sqrt{\log^2 \left(\frac{4}{x}\right)+4\pi^2 }\right)+4\pi^2}}{\log \log \left(\frac{2}{x} \right)}=1,
$$
then we can deduce that there exists a constant $M>0$ such
$$
 \frac{3}{a_n^2} \left|\log \log \left(\frac{4}{(1-\overline{\varphi(z_n)}z)} \right) \right|^2+ \frac{2}{a_n}\left| \log \log \left(\frac{4}{(1-\overline{\varphi(z_n)}z)} \right) \right|\leq M
$$
for all $n\in\Bbb N$ and all $z\in\Bbb D$. Hence, it is enough to find upper bound for
$$
 \sup_{z\in\Bbb D}\frac{v_{\log}(z)}{|1-\overline{\varphi(z_n)}z|\log \left(\frac{4}{|1-\overline{\varphi(z_n)}z|} \right)}
$$
for all $n\in\Bbb N$. Clearly, the above expression is uniformly bounded if $\frac{2}{e}\leq |1-\overline{\varphi(z_n)}z|< 2$.
Now, if $|1-\overline{\varphi(z_n)}z|<\frac{2}{e}$, then since the function $h_2(t)=t \log \left(\frac{2}{t} \right)$ is increasing on $\left(0, \frac{2}{e}\right)$, we have
$$
\frac{v_{\log}(z)}{|1-\overline{\varphi(z_n)}z|\log \left(\frac{4}{|1-\overline{\varphi(z_n)}z|} \right)}\leq\frac{ h_2(1-|z|)}{h_2\left(|1-\overline{\varphi(z_n)}z|\right)}\leq 1.
$$
Hence, Inequality (\ref{ec:1}) is bounded and there exists a constant $L>0$ such that $\|h_n \|_{\mathcal{B}^{v_{\log}}}\leq L$ for all $n \in \mathbb{N}$.
Thus, we have
\begin{eqnarray*}
\|W_{u, \varphi}(h_n) \|_{\mathcal{B}^{v_{\log}}}
&\ge& \limsup_{n \to \infty}v_{\log}\left(z_n\right)\dfrac{|u(z_n)\varphi'(z_n)|}{(1-|\varphi(z_n)|^2)\log \left(\dfrac{4}{1-|\varphi(z_n)|^2} \right)}\\
&\ge& \limsup_{n \to \infty}v_{\log}\left(z_n\right)\dfrac{|u(z_n)\varphi'(z_n)|}{2(1-|\varphi(z_n)|)\log \left(\dfrac{4}{1-|\varphi(z_n)|} \right)}\\
& \geq&C\, \limsup_{n \to \infty}v_{\log}\left(z_n\right)\dfrac{|u(z_n)\varphi'(z_n)|}{(1-|\varphi(z_n)|)\log \left(\dfrac{2}{1-|\varphi(z_n)|} \right)}\\
&=&\limsup_{n \to \infty}\dfrac{v_{\log}(z_n)}{v_{\log}(\varphi(z_n))}|u(z_n)\varphi'(z_n)|\\
&\ge& C \limsup_{n \to \infty}\dfrac{v_{3}(z_n)}{v_{3}(\varphi(z_n))}|u(z_n)\varphi'(z_n)|\\
&\ge& C \limsup_{n \to \infty}\dfrac{v_{3}(z_n)}{\widetilde{v}_{3}(\varphi(z_n))}|u(z_n)\varphi'(z_n)|
\end{eqnarray*}
where we have used the fact that $v_{\log}\simeq v_3$ and that   $v_3$ is an essential weight ($v_3$ is the weight defined in (\ref{peso})).

\vspace{0.2cm}
Since the sequence  $\left\{z_n\right\}$ such that  $\left|\varphi\left(z_n\right)\right|\to 1^-$ was arbitrary, we obtain that there exists a constant $C>0$ such that
\begin{eqnarray}\label{cota-inf-2}
\| W_{u, \varphi} \|_{e}^{\mathcal{B}^{v_{\log}}\to \mathcal{B}^{v_{\log}}}& \geq& C \limsup_{|\varphi(z)|\to 1^{-}}\frac{v_{3}(z)}{\widetilde{v}_{3}(\varphi(z))}|u(z)\varphi'(z)|\\
&\geq& C\, \| W_{u \varphi', \varphi} \|_{e}^{H_{v_{3}}^{\infty}\to H_{v_{3}}^{\infty}}\nonumber\\
&=&C\limsup_{n \to \infty}\dfrac{\|u \varphi' \varphi^n \|_{H_{v_{\log}}^{\infty}}}{\|g_n \|_{H_{v_{\log}}^{\infty}}}\nonumber\\
&=&C\limsup_{n \to \infty}\dfrac{\|I_u(\varphi^n) \|_{\mathcal{B}^{v_{\log}}}}{\|g_n \|_{\mathcal{B}^{v_{\log}}}} \nonumber.
\end{eqnarray}
Therefore, from the inequalities (\ref{cota-inf-1}) and (\ref{cota-inf-2}), we can conclude that there exists a constant $C>0$ such that
$$
\|W_{u, \varphi} \|_{e}^{\mathcal{B}^{v_{\log}}\to \mathcal{B}^{v_{\log}}}\geq C\max \left\{\limsup_{n \to \infty}\dfrac{(n+1)\| J_{u}(\varphi^n)\|_{\mathcal{B}^{v_{\log}}}}{\|g_{n+1}\|_{\mathcal{B}^{w_{\log}}}}, \limsup_{n \to \infty} \dfrac{\|I_{u}(\varphi^n)\|_{\mathcal{B}^{v_{\log}}}}{\|g_n\|_{\mathcal{B}^{v_{\log}}}} \right\}.
$$
This finishes the proof of Theorem \ref{teo1}.

\section{An application. Composition operators on Zygmund-Lo\-ga\-rithmic space}
 As an application of our results, in this section we study continuity, compactness and we estimate the essential norm of composition operators acting on Zygmund-logarithmic space. Recall that the Zygmund-logarithmic space $\mathcal{Z}^{v_{\log}}$, consists of all holomorphic functions  $f \in H(\mathbb{D})$ such that  $f' \in \mathcal{B}^{v_{\log}}$. More precisely,
$$
\mathcal{Z}^{v_{\log}}:=\{f \in H(\mathbb{D}) ~:~ \| f\|_{\widetilde{\mathcal{Z}}^{v_{\log}}}= \sup_{z \in \mathbb{D}}v_{\log}(z)|f''(z)|<\infty\}
$$
endowed with the norm   $\| f\|_{\mathcal{Z}^{v_{\log}}}:=|f(0)|+|f'(0)|+\| f\|_{\widetilde{\mathcal{Z}}^{v_{\log}}}$, $\mathcal{Z}^{v_{\log}}$ is a Banach space.

\vspace{0.2cm}
As before, for a holomorphic function $u:\Bbb D\to\Bbb C$, we define the functionals
$$
I'_{u}f(z)=\int_{0}^{z}I_{u}(f(s))ds~~~\hspace{0.2cm}\textrm{and}\hspace{0.2cm}~~~J'_{u}f(z)=\int_{0}^{z}J_{u}(f(s))ds,
$$
where  $f \in H(\mathbb{D})$ and $I_{u}, J_{u}$ are the functionals defined in Theorem \ref{cor2}. We have the relations
\begin{eqnarray*}
  \| f\|_{\widetilde{\mathcal{Z}}^{v_{\log}}} &=& \| f'\|_{\widetilde{\mathcal{B}}^{v_{\log}}}\\
  \| C_{\varphi}(f)\|_{\widetilde{\mathcal{Z}}^{v_{\log}}}&=& \|\varphi' C_{\varphi}(f) \|_{\widetilde{\mathcal{B}}^{v_{\log}}}.
\end{eqnarray*}
Hence, the operator $C_\varphi:\mathcal{Z}^{v_{\log}}\to \mathcal{Z}^{v_{\log}}$ is continuous if and only if the weighted composition operator $W_{\varphi',\varphi}$ is continuous on $\mathcal{B}^{v_{\log}}$. Thus, by Theorem \ref{cor2}, and the relations
$$
  \sup_{n \in \mathbb{W}}\frac{(n+1)\|J_{\varphi'}'(\varphi^n) \|_{\mathcal{Z}^{v_{\log}}}}{\frac{\|g_{n+2}\|_{\mathcal{Z}^{w_{\log}}}}{n+2}} =
  \sup_{n \in \mathbb{W}}\frac{(n+1)\|J_{\varphi'}(\varphi^n)\|_{\mathcal{B}^{v_{\log}}}}{\|g_{n+1}\|_{\mathcal{B}^{w_{\log}}}}<\infty\\
$$
and
$$
  \sup_{n \in \mathbb{N}}\frac{\|I_{\varphi'}'(\varphi^n) \|_{\mathcal{Z}^{v_{\log}}}}{\frac{\|g_{n+1}\|_{\mathcal{Z}^{v_{\log}}}}{n+1}} =
  \sup_{n\in \mathbb{N}}\frac{\|I_{\varphi'}(\varphi^n) \|_{\mathcal{B}^{v_{\log}}}}{\|g_n\|_{\mathcal{B}^{v_{\log}}}}< \infty.
$$
we have the following result:

\begin{cor}
Suppose that  $\varphi: \mathbb{D}\to \mathbb{D}$ is a holomorphic function. The composition operator $C_{\varphi}$ is continuous on $\mathcal{Z}^{v_{\log}}$ if and only if
$$
 \max\left\{\sup_{n\in\Bbb W}\frac{(n+2)(n+1)\left\|J_{\varphi'}'\left(\varphi^n\right)\right\|_{\mathcal{Z}^{v_{\log}}}}
  {\left\|g_{n+2}\right\|_{\mathcal{Z}^{w_{\log}}}}, \sup_{n\in\Bbb N} \frac{(n+1)\left\|I_{\varphi'}'\left(\varphi^n\right)\right\|_{\mathcal{Z}^{v_{\log}}}}
  {\left\|g_{n+1}\right\|_{\mathcal{Z}^{v_{\log}}}}\right\}<\infty.
$$
\end{cor}

Now, we go to estimate the essential norm of the continuous operator $C_{\varphi}: \mathcal{Z}^{v_{\log}}\to \mathcal{Z}^{v_{\log}}$, in fact, we go to show the following relation:
\begin{equation}\label{zygmund}
\|C_{\varphi} \|_{e}^{\mathcal{Z}^{v_{\log}} \to \mathcal{Z}^{v_{\log}}}=\| W_{\varphi', \varphi} \|_{e}^{\mathcal{B}^{v_{\log}} \to\mathcal{B}^{v_{\log}}}.
\end{equation}
The argument in the proof of Lemma \ref{le-es-rest} shows that
$$
 \|C_{\varphi} \|_{e}^{\mathcal{Z}^{v_{\log}} \to \mathcal{Z}^{v_{\log}}}=\|C_{\varphi} \|_{e}^{\widetilde{{\mathcal{Z}}}^{v_{\log}} \to \mathcal{Z}^{v_{\log}}},
$$
where $\widetilde{\mathcal{Z}}^{v_{\log}}= \left\{f \in \mathcal{Z}^{v_{\log}}:~ f(0)=f'(0)=0 \right\}$. Hence, it is enough to show that
$$
\|C_{\varphi} \|_{e}^{\widetilde{\mathcal{Z}}^{v_{\log}}\to \mathcal{Z}^{v_{\log}}}=\|W_{\varphi', \varphi} \|_{e}^{\widetilde{\mathcal{B}}^{v_{\log}}\to \mathcal{B}^{v_{\log}}}.
$$
To see this, we proceed as in the proof of Lemma \ref{lem2}, that is, we consider the derivative operator  $D: \widetilde{\mathcal{Z}}^{v_{\log}}\to \mathcal{B}^{v_{\log}}$. Since this operator is an isometry, we have the relation
\begin{equation}\label{ec:2}
    C_{\varphi}(g):= D^{-1} W_{\varphi', \varphi}D(g),
\end{equation}
for all  $g \in \mathcal{Z}^{v_{\log}}$. Thus, for  $\epsilon >0$ we can find a compact  operator  $T: \widetilde{\mathcal{B}}^{v_{\log}}\to \mathcal{B}^{v_{\log}}$ such that
\begin{eqnarray*}
\| W_{\varphi', \varphi}\|_{e}^{\widetilde{\mathcal{B}}^{v_{\log}}\to \mathcal{B}^{v_{\log}}}&\ge&
\frac{1}{1+\epsilon}\|W_{\varphi', \varphi}-T \|^{\widetilde{\mathcal{B}}^{v_{\log}}\to \mathcal{B}^{v_{\log}}}\\
&\geq&
\frac{1}{1+\epsilon}\|D^{-1}W_{\varphi', \varphi}-D^{-1}T \|^{\widetilde{\mathcal{Z}}^{v_{\log}}\to \mathcal{B}^{v_{\log}}}\\
&\ge& \frac{1}{1+\epsilon}\|D^{-1}W_{\varphi', \varphi}D-D^{-1}T D\|^{\widetilde{\mathcal{Z}}^{v_{\log}}\to \mathcal{Z}^{v_{\log}}}\\
&=&\frac{1}{1+\epsilon}\|C_{\varphi}-K \|^{\widetilde{\mathcal{Z}}^{v_{\log}}\to \mathcal{Z}^{v_{\log}}}\ge
 \frac{1}{1+\epsilon}\|C_{\varphi} \|_{e}^{\widetilde{\mathcal{Z}}^{v_{\log}}\to \mathcal{Z}^{v_{\log}}},
\end{eqnarray*}
where we have used that $D$ is an isometry and the fact that  $K:=D^{-1}T D$ is a compact operator on $\widetilde{\mathcal{Z}}^{v_{\log}}$. Therefore, since   $\epsilon$ was arbitrary,  we conclude
$$
\|C_{\varphi} \|_{e}^{\widetilde{\mathcal{Z}}^{v_{\log}}\to \mathcal{Z}^{v_{\log}}}\leq\|W_{\varphi', \varphi} \|_{e}^{\widetilde{\mathcal{B}}^{v_{\log}}\to \mathcal{B}^{v_{\log}}}.
$$
Similarly, given  $\epsilon>0$ we can find a compact operator $\widehat{T}: \widetilde{\mathcal{Z}}^{\mu_1} \to \mathcal{Z}^{\mu_2}$ such that
$$
\|C_{\varphi} \|_{e}^{\widetilde{\mathcal{Z}}^{v_{\log}} \to \mathcal{Z}^{v_{\log}}}\ge\frac{1}{1+\epsilon}\| C_{\varphi}-\widehat{T}\|^{\widetilde{\mathcal{Z}}^{v_{\log}} \to \mathcal{Z}^{v_{\log}}}.
$$
Thus, using the fact that  $D$ is an isometry  and the relation  (\ref{ec:2}), we obtain
\begin{eqnarray*}
\frac{1}{1+\epsilon}\| C_{\varphi}-\widehat{T}\|^{\widetilde{\mathcal{Z}}^{v_{\log}} \to \mathcal{Z}^{v_{\log}}}
&=&
\frac{1}{1+\epsilon}\| DC_{\varphi} -D\widehat{T}\|^{\widetilde{\mathcal{Z}}^{v_{\log}} \to\mathcal{B}^{v_{\log}}}\\
&\ge& \frac{1}{1+\epsilon}\|DC_{\varphi}D^{-1}- D\widehat{T}D^{-1}\|^{\widetilde{\mathcal{B}}^{v_{\log}} \to\mathcal{B}^{v_{\log}}}\\
&=&\frac{1}{1+\epsilon}\| W_{\varphi', \varphi}-K\|^{\widetilde{\mathcal{B}}^{v_{\log}} \to\mathcal{B}^{v_{\log}}}\ge \frac{1}{1+\epsilon} \| W_{\varphi', \varphi} \|_{e}^{\widetilde{\mathcal{B}}^{v_{\log}} \to\mathcal{B}^{v_{\log}}}.
\end{eqnarray*}
Hence, since  $\epsilon$ was arbitrary, we conclude
$$
\|C_{\varphi} \|_{e}^{\widetilde{\mathcal{Z}}^{v_{\log}} \to \mathcal{Z}^{v_{\log}}}=\| W_{\varphi', \varphi} \|_{e}^{\widetilde{\mathcal{B}}^{v_{\log}} \to\mathcal{B}^{v_{\log}}}.
$$
Therefore, we have shown the following result:

\begin{thm}
Suppose that $\varphi :\Bbb D\to \Bbb D$ is a holomorphic function and that
$ C_{\varphi}$ is continuous on  $\mathcal{Z}^{v_{\log}}$. Then
$$
\| C_{\varphi}\|_{e}^{\mathcal{Z}^{v_{\log}}\to \mathcal{Z}^{v_{\log}}}\asymp \max \left \{\limsup_{n \to \infty}\frac{(n+2)(n+1)\|J'_{\varphi'}(\varphi^n)\|_{\mathcal{Z}^{v_{\log}}}}{\|g_{n+2}\|_{\mathcal{Z}^{w_{\log}}}}, \limsup_{n \to \infty}\frac{(n+1)\|I'_{\varphi'} (\varphi^n)\|_{\mathcal{Z}^{v_{\log}}}}{\|g_{n+1} \|_{\mathcal{Z}^{v_{\log}}}} \right \}
$$
\end{thm}

As an immediate consequence of the above result, we have the following corollary:

\begin{cor}
Suppose that   $\varphi: \mathbb{D}\to \mathbb{D}$ is a holomorphic function. The composition operator  $C_{\varphi}$ is compact on $\mathcal{Z}^{v_{\log}}$ if and only if
$$
 \max\left\{\lim_{n \to \infty}\frac{(n+2)(n+1)\left\|J_{\varphi'}'\left(\varphi^n\right)\right\|_{\mathcal{Z}^{v_{\log}}}}
  {\left\|g_{n+2}\right\|_{\mathcal{Z}^{w_{\log}}}}, \lim_{n \to \infty} \frac{(n+1)\left\|I_{\varphi'}'\left(\varphi^n\right)\right\|_{\mathcal{Z}^{v_{\log}}}}
  {\left\|g_{n+1}\right\|_{\mathcal{Z}^{v_{\log}}}}\right\} = 0.
$$
\end{cor}


\end{document}